\documentclass{amsart}

\usepackage{graphicx} 
\usepackage[T1,T2A]{fontenc}
\usepackage[utf8]{inputenc}
\usepackage{amsmath}
\usepackage{bbold}
\usepackage{amssymb,latexsym}
\usepackage{xcolor}

\theoremstyle{plain}
\newtheorem{thm}{Theorem}[section]

\newtheorem{rem}[thm]{Remark}

\theoremstyle{definition}
\newtheorem{de}[thm]{Definition}
\newtheorem{exm}[thm]{Example}

\begin{document}
\title[Infinite convex geometries]{Infinite convex geometries with lower semi-modularity and join semi-distributivity}
\author[A. Mata]{Adam Mata}
  \address{Faculty of Mathematics and Information Science\\
  Warsaw University of Technology\\
  Koszykowa 75, 00-662 Warsaw, Poland}
  \email{adam.mata.dokt@pw.edu.pl}

\begin{abstract}

The following article treats about convex geometries which are lower semi-modular and join semi-distributive lattices. Firstly, it is shown that there is a class $K$ of infinite convex geometries which can be build out of finite ones by using the construction of a union of a chain. Then it is shown that elements of $K$ preserve lower semi-modularity and join semi-distributivity which are not default properties in the infinite setting. It is also discussed that not all of the infinite convex geometries may be obtained by the means of the union of a chain.   
\end{abstract}

\maketitle

\section{Introduction}
Convex geometries first occurred in \cite{Dilworth}, discovered as finite lower semimodular lattices with a unique irredundant decomposition. Their study as closure operators with the anti-exchange property is carried out in \cite{EdelmanJamison1985}. Since then they have also been studied as lattices in \cite{AGT} and \cite{KiraJB}. In \cite{CzedliSchmidt2012} the authors carried out the research concerning $SPS$ lattices which turn out to be dual to convex geometries of convex dimension 2. As long as they characterize the congruences of $SPS$ lattices (see \cite{CzedliKurusa2019}) hence we obtain this classification for $cdim = 2$ as well. The general survey on the topic of convex geometries may be found in \cite[Chapter 5]{Gratzer2023}. Convex geometries are also duals to antimatroids (see \cite{antimatroids}) and a special type of greedoids (see \cite{greedoids}). The infinite case was presented and studied in both \cite{AGT} and \cite{KiraJB-2}. It turns out that convex geometries (also infinite) even found their application in conditional logic as presented in \cite{MartiPinosio} and \cite{Marti}.

In this paper, we show how to build an infinite convex geometry out of the finite ones and reveal them as an example of class of structures which has the properties of \textit{lower semi-modularity} and \textit{join semi-distributivity}.

This setting of properties is not common in the case of infinite convex geometries as it was shown in both, accordingly, \cite[Section 1]{AGT} and \cite[Theorem 10]{KiraJB-2}.

We start the paper with some preliminaries concerning convex geometries and its representations as well as some model-theoretic tools. Then we show how to define convex geometries by $\forall\exists$ formulas. Later we present the construction of infinite convex out of finite ones as well as the fact that not all finite convex geometries are elementary equivalent. We finish with the construction of infinite convex geometries which are (JSD) and (LSM).

The idea of construction uses classical Tarski-Vaught methods of model theory involving the infinite sum of a chain of finite structures. All of the considerations exploit fundamental facts and theorems from the first-order logic setting.

\section{Preliminaries}

\begin{de}
Let $X$ be a nonempty set. We say that mapping $\alpha: 2^X \rightarrow 2^X$ is an \textbf{algebraic closure operator} if, for all $A, B \subseteq X$:
\begin{itemize}
    \item $A \subseteq \alpha(A)$,
    \item $A \subseteq B \implies \alpha(A) \subseteq \alpha(B)$,
    \item $\alpha(\alpha(A)) = \alpha(A)$,
    \item $\alpha(\varnothing) = \varnothing$.
\end{itemize}
\end{de}

\begin{de}(\cite[Section 2]{EdelmanJamison1985})
A convex geometry is a pair $\mathcal{G} = (X, \alpha)$ where:
\begin{itemize}
    \item $X$ is a nonempty set
    \item $\alpha$ is an \textbf{algebraic closure operator} with \textbf{the anti-exchange property}, i. e.:
    \begin{equation*}
        \tag{AE}
        q \in \alpha(K \cup \{p\}) \implies p \not \in \alpha(K \cup \{q\})
    \end{equation*}
\end{itemize}
where $K \subseteq X$ is $\alpha-closed$, $p,q \in X$ but $p,q \not \in K$. In general, we associate $\alpha$ with the set of all closed sets obtained by $\alpha$, which is $\{\alpha(Y):\ Y \subseteq X\}$.
\end{de}

The above definition is equivalent to the following. 

\begin{de}(\cite[Section 2.1]{Marti})
A convex geometry is a pair $\mathcal{G} = (X, \alpha)$ where $X$ is a nonempty set and $\alpha$ is a family of subsets of $X$ such that:
\begin{itemize}
    \item $\alpha$ is \textbf{closed under intersection}, i. e. for any $\beta \subseteq \alpha$ it holds:
    $$
    \bigcap \beta \in \alpha
    $$
    \item $\alpha$ has an \textbf{anti-exchange property}, i. e. for every $A \in \alpha$ and every $x, y \in X$ with $x, y \in X$ and $x \not = y$ there is $B \in \alpha$ with $A \subseteq B$ such that $x \in B$ and $y \not \in B$, or $x \not \in B$ and $y \in B$.
\end{itemize}

\end{de}

In case of finite convex geometries, a set of closed sets under intersection may be also introduced by an idea of downsets of finite generating chains. This notion is introduced and developed in \cite{EdelmanJamison1985}. It was proven in this paper, that any finite convex geometry may be obtained by the following construction:

Let's take a finite set $X = \{x_1, x_2, \dots x_n\}$. We take $K$ different chains composed out of elements from $X$:

\begin{itemize}
    \item $C_1 = x_{i^1_1} < x_{i^1_2} < \dots < x_{i^1_n}$,
    \item $C_2 = x_{i^2_1} < x_{i^2_2} < \dots < x_{i^2_n}$,
    \item $\dots$
    \item $C_K = x_{i^K_1} < x_{i^K_2} < \dots < x_{i^K_n}$.
\end{itemize}

Let $\mathfrak{D}(C_i)$ denote a family of downsets of the chain $C_i$. For example:

\begin{equation}
    \mathfrak{D}(C_1) = \{\varnothing, \{x_{i^1_1}\}, \{x_{i^1_1}, x_{i^1_2}\}, \dots, X\}
\end{equation}

Set $\mathfrak{D} = \bigcup_{i = 1}^K \mathfrak{D}(C_i)$. Let $\mathfrak{D}^{\cap}$ be a set $\mathfrak{D}$ closed under intersection. Then $\mathcal{G} = (X, \mathfrak{D}^{\cap})$ is a convex geometry. A minimal number of chains needed to generate a particular convex geometry is called a convex dimension and is denoted as $cdim(\mathcal{G})$.

We exploit this schema of construction further in the paper.
\section{Universal and $\forall\exists$ Formulas - Preservation by a Union of a Chain}

\begin{de}
A universal formula is a formula in a form:

\begin{equation*}
    \forall x_1 \forall x_2 \dots \forall x_n: A(x_1, x_2, \dots x_n), 
\end{equation*}

where $A(x_1, x_2, \dots x_n)$ is a formula which does not contain any quantifier and $n$ may be equal to $0$. Further, a $\forall\exists$ formula is formula in the form:

\begin{equation*}
    \forall x_1 \forall x_2 \dots \forall x_n \exists y_1 \exists y_2 \dots \exists y_m: A(x_1, x_2, \dots x_n, y_1, y_2, \dots y_m),
\end{equation*}

where $A(x_1, x_2, \dots x_n, y_1, y_2, \dots y_m)$ is a formula which does not contain any quantifier and $n$ or $m$ (or both) may be equal to $0$.
\end{de}

Let us notice that every universal formula is an instance of a $\forall\exists$ formula. Further, we introduce an idea of a union of a chain of models. The following facts may be found in \cite[Chapter 8]{CoriLascar}.
\subsection{A union of a chain of models}

\begin{de}
\cite[Definition 8.1]{CoriLascar}
Let $L$ be a first order language, let $\mathcal{M}$ be an $L$-structure and let $\mathcal{N}$ be a substructure of $\mathcal{M}$. We say that $\mathcal{N}$ is an elementary substructure of $\mathcal{M}$ if, for every formula $F(x_1, x_2, \dots x_k)$ and elements $n_1, n_2, \dots n_k \in N$ we have:

\begin{equation*}
\mathcal{M} \models F(n_1, n_2, \dots n_k) \iff \mathcal{N} \models F(n_1, n_2, \dots n_k).
\end{equation*}

\end{de}

\begin{de}
\cite[Definition 8.10]{CoriLascar}
A function $h: M \rightarrow N$ is an elementary mapping if, for every formula $F(x_1, x_2, \dots x_k)$ and elements $m_1, m_2, \dots m_k \in M$ we have:

\begin{equation*}
\mathcal{M} \models F(m_1, m_2, \dots m_k) \iff \mathcal{N} \models F(h(m_1), h(m_2), \dots h(m_k)).
\end{equation*}

In such a case we say that $\mathcal{M}$ can be elementary embedded into $\mathcal{N}$.

\end{de}

Remarkable observation is that any elementary mapping is in fact an embedding. It is sufficient to consider a formula $A(x_1, x_2) = x_1 \simeq x_2$ to notice so. Let $h$ be a monomorphism from $\mathcal{M}$ into $\mathcal{N}$. It is an elementary mapping if and only if the image of $h$ is an elementary substructure of $\mathcal{N}$.

Let $\mathcal{M}$ be an $L$-structure. The complete theory of $\mathcal{M}$, denoted by $Th(\mathcal{M})$, is a set of formulas:

\begin{equation*}
    Th(\mathcal{M}) = \{\phi: \phi \text{ is a closed formula of $L$ and } \mathcal{M} \models \phi\}.
\end{equation*}

If $\mathcal{M}$ and $\mathcal{N}$ are both $L$-structures, we say that $\mathcal{M}$ and $\mathcal{N}$ are elementary equivalent if $Th(\mathcal{M}) = Th(\mathcal{N})$, what is denoted by $\mathcal{M} \equiv \mathcal{N}$. The following theorem gives a tool to discover whether two structures are elementary equivalent.

\begin{thm}
\label{equivthm}
\cite[Theorem 8.22]{CoriLascar}
Let $\mathcal{M}_1$ and $\mathcal{M}_2$ be $L$-structures. $\mathcal{M}_1 \equiv \mathcal{M}_2$ if and only if there is some third $L$-structure into which both $\mathcal{M}_1$ and $\mathcal{M}_2$ can be embedded. 
\end{thm}

Let us consider an infinite, totally-ordered set of indices $(I, <)$, such that for every $i \in I$, $\mathcal{M}_i$ is an $L$-structure and if $i < j$ then $\mathcal{M}_i$ is a substructure of $\mathcal{M}_j$. Set $M = \bigcup_{i \in I} M_i$. Then we can create a structure $\mathcal{M}$ also denoted as $\bigcup_{i \in I} \mathcal{M}_i$, such that for every $i \in I$, $\mathcal{M}_i$ is a substructure of $\mathcal{M}$. Furthermore for any relational symbol $R \in L$ and $m_1, m_2, \dots m_k \in M_i$:

\begin{equation*}
(m_1, m_2, \dots m_k) \in R^{\mathcal{M}_i} \iff (m_1, m_2, \dots m_k) \in R^{\mathcal{M}}. 
\end{equation*}

Constructing $\mathcal{M}$, we proceed with function and constant symbols in the same manner.

\begin{de}
\cite[Definition 8.42]{CoriLascar}
Let $T$ be a theory in a language $L$. $T$ is preserved by a union of chains if every union of a chain of models of $T$ is a model of $T$.    
\end{de}

The following theorem is a direct consequence of the previous definition.

\begin{thm}
\label{e-map-sum-thm}
Let $(I, <)$ be a totally-ordered set and, for every $i \in I$, let $\mathcal{M}_i$ be an $L$-structure. Further, assume that if $i < j$ then $\mathcal{M}_i$ is an elementary substructure of $\mathcal{M}_j$. Then for every $i \in I$, $\mathcal{M}_i$ is an elementary substructure of $\mathcal{M}$. 
\end{thm}

\begin{thm}
\label{union-ae-thm}
\cite[Theorem 8.43]{CoriLascar}
A theory $T$ is preserved by a union of chains if and only if $T$ is equivalent to some $\forall\exists$ theory (a theory composed of $\forall\exists$ formulas only).
\end{thm}

\section{Defining finite Convex Geometries by Formulas}
\label{defining_section}

Every finite convex geometry may be considered to be a lattice $\mathcal{G} = (G, \wedge, \vee)$ with two additional conditions imposed, namely \textit{lower semi-modularity} (LSM) and \textit{join semi-distributivity} (JSD) (see \cite[Theorem 1.9]{AGT}). Hence, it is defined by the following formulas:

\begin{itemize}
    \item[(1)] $\forall x : x \vee x = x$,
    \item[(2)] $\forall x : x \wedge x = x$,
    \item[(3)] $\forall x \forall y: x \vee y = y \vee x$,
    \item[(4)] $\forall x \forall y: x \wedge y = y \wedge x$,
    \item[(5)] $\forall x \forall y: (x \vee y) \wedge y = y$,
    \item[(6)] $\forall x \forall y: (x \wedge y) \vee y = y$,
    \item[(7)] $\forall x \forall y \forall z: (x \vee y) \vee z = x \vee (y \vee z)$,
    \item[(8)] $\forall x \forall y \forall z: (x \wedge y) \wedge z = x \wedge (y \wedge z)$,
    \item[(JSD)] $\forall x \forall y \forall z: (x \vee y = x \vee z) \implies (x \vee (y \wedge z) = (x \vee y) \wedge (x \vee z))$,
    \item[(LSM)]$\forall x \forall y: (x \prec (x \vee y)) \implies ((x \wedge y) \prec y)$,
\end{itemize}

where $\prec$ denotes covering relation.

Formulas (1)-(8) and (JSD) are in the universal form. Below, we notice that (LSM) also may be transformed into the universal form. Let us recall that:

\begin{equation}
\tag{$\alpha$}
\forall x \forall y: x \leq y \iff x \vee y = y,     
\end{equation}

as well as:

\begin{equation}
\tag{$\beta$}
\forall x \forall y: x \prec y \iff ((x \not = y)\ \&\ (x \leq y)\ \&\ (\forall z: (x \leq z \leq y) \implies (z = x\ ||\ y = z))).
\end{equation}

Then, using the presented steps, (LSM) may be translated to:
\begin{equation*}
\tag{LSM'}
\forall x \forall y \forall z \forall r: \Phi(x, y, z, r),
\end{equation*}

where $\Phi(x, y, z, r)$ does not contain any quantifier.

Since all of the formulas which define the class of convex geometries are universal formulas and, in particular, $\forall\exists$ formulas then the class of convex geometries is closed under the operation of the union of a chain.

\section{Constructing infinite Convex Geometries out of finite ones}
\label{chapter_4}

The alternative way of representing a finite convex geometry is its generation by finite chains. We use the approach which exploits generating convex geometries by chains defined on a particular finite set (\cite{EdelmanJamison1985}). Convex dimension of a geometry is a minimal number of different chains necessary to generate the geometry. We show by example how to create an infinite convex geometry into which we embed infinite number of finite convex geometries. Let us assume that $I$ is an enumerable set of indices. Finite convex geometries are sequenced in a chain in such a manner that for $i < j$, $G_i$ is a subgeometry of $G_j$.

Let us start with a finite convex geometry, denoted by $\mathcal{G}_1 = (X_1, \alpha_1)$ of convex dimension $K_{\mathcal{G}_1}$. Furthermore let $|X_1|=N_1$. Then $\mathcal{G}_1$ is generated by the set of chains:

\begin{itemize}
    \item $a_{11} < a_{12} < \dots < a_{1N_1}$
    \item $a_{21} < a_{22} < \dots < a_{2N_1}$
    \item $\dots$
    \item $a_{K_{\mathcal{G}_1}1} < a_{K_{\mathcal{G}_1}2} < \dots < a_{K_{\mathcal{G}_1}N_1}$
\end{itemize}
where $a_{ij} \in X$.

To create the next convex geometry we introduce a new element ($X_1 = X \cup \{b\}$) and refactor the generating chains of $\mathcal{G}_1$ in the following manner:

\begin{itemize}
    \item $a_{11} < a_{12} < \dots < a_{1N_1} < b$
    \item $a_{21} < a_{22} < \dots < a_{2N_1} < b$
    \item $\dots$
    \item $a_{K_{\mathcal{G}_1}1} < a_{K_{\mathcal{G}_1}2} < \dots < a_{K_{\mathcal{G}_1}N_1} < b$
    \item $b < a_{1N_1} < a_{1(N_1 - 1)} < \dots < a_{11}$  
\end{itemize}

Taking the family of downsets of such chains and closing it under intersection we obtain a convex geometry $\mathcal{G}_2 = (X_2, \alpha_2)$. Let us notice that in this case $\alpha_1 \subsetneq \alpha_2$. It yields two significant results:

\begin{itemize}
    \item The last chain of $\mathcal{G}_2$ generates $\{b\} \in \alpha_2$ which is not a result of any other downset nor intersection of downsets in the previous chains. Hence the successor geometry is significantly different as the previous one and has convex dimension greater by 1 comparing to the predecessor.
    \item The predecessor geometry is embeddable into the successor one since $\alpha_1 \subseteq \alpha_2$. It is sufficient to pick an embedding $\mathcal{E}: \alpha_1 \hookrightarrow \alpha_2$ which $C \overset{\mathcal{E}}{\mapsto} C$.
\end{itemize}

Proceeding in the same manner we obtain an infinite chain of finite convex geometries in which each predecessor is embeddable in the successor.

Since the class of convex geometries is equivalent to a $\forall\exists$ theory, then by \ref{union-ae-thm}, we obtain that the class of convex geometries is closed under union of a chain. Having a sequence of convex geometries as presented above, denoted $\mathcal{G}_1, \mathcal{G}_2, \mathcal{G}_3, \dots$, by a union we obtain a convex geometry:

\begin{equation*}
    \mathcal{G} = \bigcup_{i \in \omega} \mathcal{G}_i.
\end{equation*}

\begin{exm}
Let $\mathcal{G}_1$ be a convex geometry generated by the following chains:

\begin{itemize}
    \item $a < b < c$,
    \item $b < a < c$.
\end{itemize}

Then $\mathcal{G}_1 = (\{a, b, c\}, \{\varnothing, a, b, ab, abc\})$.

Further, let's introduce $\mathcal{G}_2$ which is generated by:

\begin{itemize}
    \item $a < b < c < d$,
    \item $b < a < c < d$,
    \item $d < a < b < c$,
\end{itemize}

which generates $\mathcal{G}_2 = (\{a, b, c, d\}, \{\varnothing, a, b, d, ab, ad, abc, abd, abcd\})$.

Again, let us introduce $\mathcal{G}_3$:
\begin{itemize}
    \item $a < b < c < d < e$,
    \item $b < a < c < d < e$,
    \item $d < a < b < c < e$,
    \item $e < a < b < c < d$,
\end{itemize}

which gives $\mathcal{G}_3 = (\{a, b, c, d, e\}, \{\varnothing, a, b, d, e, ab, ad, ae, abc, abd, abe, abcd, abce, abcde\})$.

$\mathcal{G}_1$ is clearly embeddable in $\mathcal{G}_2$ and $\mathcal{G}_2$ is embeddable in $\mathcal{G}_3$. Proceeding in the same manner for denumerably many steps we obtain a chain of convex geometries where a former geometry is embeddable in the succeeding one. Taking a union of this chain we also obtain a convex geometry since convex geometry theory is $\forall\exists$.

\end{exm}

\begin{de} \cite[Chapter 7, Section 4]{burbaki}
    Let $(I, \leq)$ be a directed set - a set where for any finite subset of $I$ there exists its upper bound within $I$. Let $\{A_i\ :\ i \in I\}$ be a family of algebras of the same type indexed by $I$. Further, let $f_{ij} : A_i \rightarrow A_j$ be a homomorphism for all $i \leq j$ with the following properties:
    \begin{itemize}
        \item $f_{ii}$ is an identity on $A_i$,
        \item For all $i, j, k \in I$ such that $i \leq j \leq k$ it holds that: $f_{ik} = f_{jk} \circ f_{ij}$.
    \end{itemize}
    then the pair $(A_i, f_{ij})$ is called a direct system over $I$.
    
    Let us introduce an equivalence relation $\sim$ on $A_i \times A_j$. Let $x_i \in A_i$ and $a_j \in A_j$ then $a_i \sim a_j$ if and only if, there is some $k \in I$ such that $i \leq k$ and $j \leq k$ and $f_{ik}(x_i) = f_{jk}(x_j)$.

    The direct limit of the \textbf{direct system} $(A_i, f_{ij})$ is the disjoint union:
    $$
    \underset{i \in I}{\bigsqcup} A_i {\bigg /} \sim
    $$
    The disjoint union of a direct system $(A_i, f_{ij})$ is denoted by $\underset{\longrightarrow}{lim}\ A_i$.
\end{de}

\begin{rem}
Let us notice that union of a chain $\underset{i \in \omega}{\bigcup} \mathcal{G}_i$ is a special case of direct limit of a sequence of structures $\underset{\longrightarrow}{lim}\ \mathcal{G}_i$ where underlying directed set of indices $\langle I, \leq\rangle$ is simply an infinite chain. Hence, in general, all of the operations, relations and constants present in the component structures $\mathcal{G}_i : i \in I$ are obtained in the result structure in the same manner as in the case of direct limit.
\end{rem}

\section{Not all convex geometries are elementary equivalent}

In this section we show that not all finite convex geometries are elementary equivalent even though we are able to construct an infinite sequence of finite convex geometries where the predeceasing one is elementary equivalent to the consequent one.

The following theorem is known as a \textit{Tarski-Vaught Test} and also may be found in \cite[Chapter 8]{CoriLascar} together with its proof.

\begin{thm}[Tarski-Vaught Test]
    \label{tarski-vaught-test}
    Let $\mathcal{M}$ be a structure and $\mathcal{N}$ be the a substructure of $\mathcal{M}$. Assume that for every formula $A(v_0, v_1, \dots, v_n)$ in language $\mathfrak{L}$ and for all elements $a_1, a_2, \dots a_n \in N$ it holds that if:
    $$
    \mathcal{M} \models \exists v_0: A(v_0, a_1, a_2, \dots a_n)
    $$
    then there is $a_0 \in N$ such that:
    $$
    \mathcal{M} \models A(a_0, a_1, a_2, \dots a_n).
    $$
    Then $\mathcal{N} \hookrightarrow \mathcal{M}$.
\end{thm}

Let $\mathcal{H}$ and $\mathcal{G}$ be finite convex geometries in  language $\mathfrak{L} = \{\vee, \wedge\}$ such that $\mathcal{G}$ is a substructure of $\mathcal{H}$.

Since $\mathcal{G}$ is finite we can construct a formula $A(\mathcal{G})$ as a conjuction of all atomic formulas in $\mathcal{G}$ written with variables $v_g$ where $g \in G$. If $g_1 \vee g_2 = g_3$ in $\mathcal{G}$ then we construct a formula $x_{g_1} \vee x_{g_2} = x_{g_3}$. We proceed for $\wedge$ accordingly.

\begin{rem}
    Further in this section, for any two variables $v_0$ and $v_1$, we use a formula:
    $$
    v_0 > v_1
    $$
    to denote the fact that:
    $$
    v_0 \vee v_1 = v_0\ \&\ \neg(v_0 = v_1),
    $$
    where $\&$ denotes logical conjunction.
\end{rem}

By the construction of $A(\mathcal{G})$ we obtain that there is an element $h \in H$ such that for all $g \in G$ it holds that $h > g$. Therefore formula:

$$
\psi(v_g : g \in G) = \exists x: \phi(x, v_g : g \in G)
$$

where

\[
\phi(x, v_g : g \in G) = A(\mathcal{G})\ {\tiny\&} \underset{g \in G}{\text{\fontsize{15pt}{24pt}\selectfont\&}}(x > v_g)
\]

holds in $\mathcal{H}$ when all variables $v_g$ are replaced by elements $g \in G$. But there is no element in $G$ to replace variable $x$ in this formula $\psi(v_g : g \in G)$ so the formula still holds in $\mathcal{H}$. Hence, the Tarski-Vaught test fails for $\mathcal{G} \hookrightarrow \mathcal{H}$.

\section{Preserving (JSD) and (LSM) by Union of Chains}

Both (JSD) and (LSM) are still valid in the union of finite convex geometries, due to the fact that these properties may be expressed by $\forall\exists$ formulas. This is not obvious in the infinite settings. As discussed in \cite{AGT}, infinite convex geometries, when defined by the anti-exchange property of their closure operator, may not satisfy (JSD).

In case of (LSM) property, some infinite convex geometries may not even have any single cover among their closed sets, like in the case of $\mathbb{R}$ together with operations $min$ and $max$ which is also a convex geometry. This is also true in elementary equivalent chain-geometry $\mathbb{Q}$. Thus (LSM)
becomes vacuous.

The direction of studies in \cite{AGT} and \cite{KiraJB-2} are towards
the classes of infinite convex geometries that mimic properties of finite geometries, in particular, they have properties that allow them to have a lot of covers (algebraic, co-algebraic, weakly atomic, strongly coatomic, spatial, dually spatial etc).

But even in these cases the lower semimodularity may not hold. The following example is given in \cite[Discussion after Theorem 10]{KiraJB-2}. The authors provide an example of an infinite convex geometry $(\omega + 1)^d \times2$ with its atom doubled. This lattice is strongly coatomic, strongly spatial, locally distributive, and even satisfies (JSD), but it is not lower semimodular.

Thus, the unions of chains of finite convex geometries represent a proper subclass of all infinite geometries. 

\begin{thm}
    Every infinite chain of finite convex geometries such that $\mathcal{G}_1 \hookrightarrow \mathcal{G}_2 \hookrightarrow \mathcal{G}_3 \hookrightarrow \dots \hookrightarrow \mathcal{G}_i \hookrightarrow \dots$ when summed results in infinite convex geometry with both (JSD) and (LSM) property.
\end{thm}

\begin{proof}
Let us assume that all of the convex geometries in a chain:
$\mathcal{G}_1 \hookrightarrow \mathcal{G}_2 \hookrightarrow \mathcal{G}_3 \hookrightarrow \dots \hookrightarrow \mathcal{G}_i \hookrightarrow \dots$
are finite for all $i \in I$. Then, by Theorem \ref{union-ae-thm}, $\underset{i \in I}{\bigcup}\mathcal{G}_i$ satisfies (JSD) and (LSM).
\end{proof}

\section{Further research}

An interesting direction for future research would be to investigate whether the class of infinite convex geometries obtained via unions of finite ones admits a characterization by forbidden configurations, analogous to the finite obstruction sets identified for chains of ideals in ordered sets as presented in ~\cite{scattered_old_3}.

Further, another possible direction for research involves exploring the interplay between convex geometries and the theory of scattered posets, particularly in the context of their completions. The results presented in \cite{scattered_old_2} characterize when the MacNeille completion of a poset $P$ is scattered, via a finite set of forbidden substructures. This raises the question:

\begin{quote}
\emph{Can the class of infinite convex geometries obtained via unions of finite ones be characterized similarly by a finite set of forbidden substructures or embeddings?}
\end{quote}

The method of constructing infinite convex geometries by unions of chains of finite ones preserves several properties: lower semi-modularity and join semi-distributivity. These constructions are inherently well-behaved, and it remains open whether a MacNeille-type characterization could distinguish these 'constructible' infinite geometries from the other ones.

In particular, developing an analogue of  in \cite[Theorem 1.3]{scattered_old_2} for infinite convex geometries could yield a new class of forbidden configurations or closure-theoretic obstructions, expanding our understanding of the structure and limits of a convex geometry.

\begin{quote}
\emph{Can every infinite convex geometry arising via chain unions be realized as the closure system of some multichain structure, perhaps generalized to transfinite or scattered index sets?}
\end{quote}

This would bridge the constructibility of infinite convex geometries by union-of-chain construction with existing representations applied in \cite{scattered_new}.

\section{Acknowledgements}

I would like to express the highest gratitude to Prof. Kira Adaricheva from Hofstra University for her comments on the earlier versions of this manuscript. I would like to thank Warsaw University of Technology for providing the grant for visiting Hofstra University on this project in the Fall of 2023 within the SEED - Smart Education for Engineering Doctors project under STER - Internationalization of doctoral schools programme financed by the Polish National Agency for Academic Exchange (Agreement no. PPI/STE/2020/1/00018/U/00001 dated 21.12.2020).

\end{document}